\documentclass{article}

\usepackage{amsfonts,
amsmath,
amsthm,
amssymb,
url,
hyperref,
MnSymbol,
enumerate,
color}
\hypersetup{colorlinks=true,
urlcolor=blue,
linkcolor=cyan}
\usepackage[top=1in,bottom=1in,left=1in,right=1in]{geometry}

\newcommand{\F}{\ensuremath{{\mathbb{F}}}}

\newcommand{\R}{\ensuremath{{\mathbb{R}}}}
\newcommand{\Z}{\ensuremath{{\mathbb{Z}}}}
\newcommand{\lv}{\ensuremath{\left\vert}}
\newcommand{\rv}{\ensuremath{\right\vert}}
\newcommand{\lp}{\ensuremath{\left(}}
\newcommand{\rp}{\ensuremath{\right)}}
\newcommand{\lb}{\ensuremath{\left\{}}
\newcommand{\rb}{\ensuremath{\right\}}}

\DeclareMathOperator{\coker}{coker}

\DeclareMathOperator{\Surj}{Surj}
\DeclareMathOperator{\Inj}{Inj}

\DeclareMathOperator{\Aut}{Aut}
\DeclareMathOperator{\sub}{sub}
\DeclareMathOperator{\rank}{rank}

\makeatletter
\newtheorem*{rep@theorem}{\rep@title}
\newcommand{\newreptheorem}[2]{%
\newenvironment{rep#1}[1]{%
 \def\rep@title{\hyperref[##1]{#2 \ref*{##1}}}%
 \begin{rep@theorem}}%
 {\end{rep@theorem}}}
\makeatother

\theoremstyle{definition}
\newtheorem{Not21}{Definition}[section]
\newtheorem{Not22}[Not21]{Notation}
\newtheorem{Not23}[Not21]{Notation}

\theoremstyle{plain}
\newtheorem{final}[Not21]{Corollary}
\newtheorem{notmine}[Not21]{Theorem}
\newtheorem{trivialyes}[Not21]{Corollary}
\newtheorem{inj}[Not21]{Proposition}
\newtheorem{subg}[Not21]{Lemma}
\newtheorem{subg2}[Not21]{Lemma}
\newtheorem{subg2.5}[Not21]{Lemma}
\newtheorem{subg3}[Not21]{Theorem}
\newtheorem{temp}[Not21]{Theorem}
\newreptheorem{theorem}{Theorem}
\newreptheorem{corollary}{Corollary}

\theoremstyle{remark}
\newtheorem{amalgam}[Not21]{Remark}
\newtheorem{surjinj}[Not21]{Remark}

\title{A weighted M\"{o}bius function}
\author{Derek Garton}
\date{\today}

\begin{document}
\maketitle

\section{Introduction} \label{introduction}

Fix an odd prime $\ell$ and let $\mathcal{G}$ be the poset of isomorphism classes of finite abelian $\ell$-groups, with the relation $\left[A\right]\leq\left[B\right]$ if and only if there exists an injective group homomorphism $A\hookrightarrow B$.
(For notational simplicity, from this point forward we will conflate finite abelian $\ell$-groups and the equivalence classes containing them.)
In 1984, Cohen and Lenstra~\cite{CL} proved that the function\begin{align*}
\nu:\mathcal{G}&\to\R^{\geq0}\\
A&\mapsto\lv\Aut{A}\rv^{-1}\prod_{i=1}^\infty{\lp1-\ell^{-i}\rp}
\end{align*}
is a discrete probability distribution on $\mathcal{G}$.
(This fact had already been proved by Hall in~\cite{Hall}, who used a different method).
They then conjectured that if $A\in\mathcal{G}$, then $\nu(A)$ is the probability that the $\ell$-Sylow subgroup of the ideal class group of an imaginary quadratic number field is isomorphic to $A$.
Since then, mathematicians have defined various probability distributions on $\mathcal{G}$ and conjectured that these distributions describe various phenomena, both number-theoretic (e.g., \cite{FW}, \cite{CM}, \cite{EVW}, \cite{Mal}, \cite{G}) and combinatorial (e.g., \cite{MMW}, \cite{CKLPM}).

Given any discrete probability distribution $\xi:\mathcal{G}\to\R^{\geq0}$ and any $A\in\mathcal{G}$, define the \emph{$A$th moment of $\xi$} to be
\[
\sum_{B\in\mathcal{G}}{\lv\Surj{(B,A)}\rv\xi(B)},
\]
where for any $B,A\in\mathcal{G}$, we define $\Surj{(B,A)}$ to be the set of surjective group homomorphisms from $B$ to $A$.
This terminology, which is becoming standard in the literature related to the Cohen-Lenstra heuristics (see, for example,~\cite{EVW} and~\cite{MMW}), is meant to evoke an analogy with the $k$th moment of a real-valued random variable $X$: just as the $k$th moment of $X$ is the expected value of $X^k$, the $A$th moment of $\xi$ is the expected value of $\lv\Surj{\lp B,A\rp}\rv$, where $B$ is a $\mathcal{G}$-valued random variable.
Moreover, under certain favorable conditions, the set of $A$th moments of a distribution on $\mathcal{G}$ determines the distribution, making the analogy even stronger.

A precise description of these ``favorable conditions", however, is still elusive.
In~\cite{EVW}, \cite{MMW}, and \cite{G}, for example, the moments of the particular discrete probability distributions on $\mathcal{G}$ in question completely determine the distribution.
In~\cite{G}, a M\"{o}bius inversion-type procedure transforms closed formulas for moments of certain distributions on $\mathcal{G}$ into closed formulas for the distribution itself.
Some natural questions are:
\begin{itemize}
\item what is this M\"{o}bius-type function?,
\item in what ways does it behave like the classical M\"{o}bius function?, and
\item in what conditions can it transform formulas for moments into formulas for distributions?
\end{itemize}

In this paper, we focus on the first two questions, leaving the third for later work.
In \hyperref[defs]{Section~\ref*{defs}}, we begin by addressing the first question.
That is, we define this new M\"{o}bius-type function associated to the poset $\mathcal{G}$, which we denote $S:\mathcal{G}\times\mathcal{G}\to\Z$.
We also compare it to the case of the poset of subgroups of a group $G$, which we denote $\mathcal{P}_G$, and its associated M\"obius function, which we denote $\mu_G:\mathcal{P}_G\times\mathcal{P}_G\to\Z$.
In particular, we state a result relating these two functions; see \hyperref[amalgam]{Remark~\ref*{amalgam}}.
We then state the main results of the paper, Theorems~\ref{temp} and~\ref{subg3}, which we prove in \hyperref[Proofs]{Section~\ref*{Proofs}}.
As an example application of Theorems~\ref{temp} and~\ref{subg3}, we remark that they immediately imply:
\begin{repcorollary}{final}
If $A,C\in\mathcal{G}$, then $S(A,C)=0$ unless there exists an injection $\iota:A\hookrightarrow B$ with $\coker{(\iota)}$ elementary abelian.
\end{repcorollary}
\noindent We would like to note the analogy between \hyperref[final]{Corollary~\ref*{final}} and Hall's result from 1934~\cite{Hall2}: if $G$ is an $\ell$-group of order $\ell^n$, then $\mu_G(1,G)=0$ unless $G$ is elementary abelian, in which case $\mu_G(1,G)=(-1)^n\ell^{\binom{n}{2}}$.
In addition to implying \hyperref[final]{Corollary~\ref*{final}}, Theorems~\ref{temp} and~\ref{subg3} are both integral to the inversion procedure deployed in~\cite{G}, and will be a useful tool in answering the third question mentioned above.
In~\cite{Gab}, we explore further properties of $S$, using it to expand on Cohen-Lenstra's identities on finite abelian $\ell$-groups~\cite{CL}.
Moreover, \hyperref[final]{Corollary~\ref*{final}} has applications to recent work in group theory: see Lucchini's \hyperref[notmine]{Theorem~\ref*{notmine}}~\cite{Luc}, below, and the discussion following it.

\section{Definitions and results} \label{defs}

Let $\mathcal{P}$ be a locally finite poset.
The \emph{M\"{o}bius function} on $\mathcal{P}$, denoted by $\mu_\mathcal{P}$, is defined by the following criteria: for any $x,z\in\mathcal{P}$,
\begin{align*}
\hspace{73px}&\mu_\mathcal{P}\lp x,z\rp&&\hspace{-83px}=0&&
\hspace{-73px}\text{ if }x\nleq z,\\
\hspace{73px}&\mu_\mathcal{P}\lp x,z\rp&&\hspace{-83px}=1&&
\hspace{-73px}\text{ if }x=z,\\
\hspace{73px}\sum_{x\leq y\leq z}
&{\mu_\mathcal{P}\lp x,y\rp}&&\hspace{-83px}=0&&
\hspace{-73px}\text{ if }x<z.
\end{align*}
A classic reference for M\"{o}bius functions is~\cite{Rota}.
Now, for any finite group $G$, let $\mathcal{P}_G$ be the poset of subgroups of $G$ ordered by inclusion.
(To ease notation, let $\mu_G$ be the M\"{o}bius function on this poset.)
For a history of the work on the M\"{o}bius function on this particular poset, see~\cite{HIO}.
Recall that $\mathcal{G}$ is the poset of isomorphism classes of finite abelian $\ell$-groups.

\begin{Not21} \label{Not21}
For any $A,C\in\mathcal{G}$, let $\sub{(A,C)}$ be the number of subgroups of $C$ that are isomorphic to $A$.
If $A\in\mathcal{G}$, an \emph{$A$-chain} is a finite linearly ordered subset of $\lb B\in\mathcal{G}\mid B>A\rb$.
Now, given an $A$-chain $\mathfrak{C}=\lb A_j\rb_{j=1}^i$, define
\[
\sub{\lp\mathfrak{C}\rp}
:=(-1)^i\sub{\lp A,A_1\rp}\prod_{j=1}^{i-1}{\sub{\lp A_j,A_{j+1}\rp}}.
\]
Finally, for any $A,C\in\mathcal{G}$, let
\[
S(A,C)
=\begin{cases}
\hspace{18px}0&\text{if }A\nleq C,\\
\hspace{18px}1&\text{if }A=C,\\
\displaystyle{
\sum_{\substack{A\text{-chains }\mathfrak{C},\\
\max{\mathfrak{C}}=C}}
{\sub{\lp\mathfrak{C}\rp}}}&\text{if } A<C.
\end{cases}
\]
\end{Not21}

\begin{amalgam} \label{amalgam}
Though $S$ is defined on the poset $\mathcal{G}$, it is closely related to the classical work on the M\"{o}bius function on the poset of subgroups of a fixed group.
Indeed, by applying Lemma~2.2 of~\cite{HIO}, we see that
\[
S(A,C)
=\sum_{\substack{B\leq C
\\B\simeq A}}
{\mu_C(B,C)}.
\]
\end{amalgam}
%

There has been recent progress towards describing groups with non-zero M\"{o}bius functions.
For example, in 2007 Lucchini~\cite{Luc} proved the following:

\begin{notmine} \label{notmine}
Assume that $G$ is a finite solvable group and that $H$ is a proper subgroup of $G$ with $\mu_G(1,H)\neq0$.
Then there exists a family $M_1,\ldots,M_t$ of maximal subgroups of $G$ such that
\begin{itemize}
\item $H=M_1\cap\cdots\cap M_t$, and
\item $\left[G:H\right]=\left[G:M_1\right]\cdots\left[G:M_t\right]$.
\end{itemize}
\end{notmine}
\noindent In the light of \hyperref[amalgam]{Remark~\ref*{amalgam}}, \hyperref[final]{Corollary~\ref*{final}} implies that there exists an infinite family of pairs of finite abelian $\ell$-groups with trivial M\"{o}bius function:
\begin{repcorollary}{trivialyes}
If $A,C\in\mathcal{G}$ and $C$ has exactly one subgroup isomorphic to $A$, then $\mu_C(A,C)=0$ unless there exists some $\iota:A\hookrightarrow C$ with $\coker{(i)}$ elementary abelian.
\end{repcorollary}

In \hyperref[Proofs]{Section~\ref*{Proofs}}, below, we prove the main results of this paper, mentioned in \hyperref[introduction]{Section~\ref*{introduction}}.
(See \hyperref[Not23]{Notation~\ref*{Not23}} for the definition of rank.)

\begin{reptheorem}{temp}
Suppose that $A,C\in\mathcal{G}$ and $\rank{A}<\rank{C}$.
If there exists $k\in\Z^{>0}$ and $B\in\mathcal{G}$ such that $A\leq B<C$, $\rank{B}=\rank{A}$, and
\[
B\oplus\overbrace{\lp\Z/\ell\rp\oplus\cdots\oplus\lp\Z/\ell\rp}^{k\text{ times}}=C,
\]
then $S(A,C)=S(A,B)\cdot S(B,C)$.
Otherwise, $S(A,C)=0$.
\end{reptheorem}

\begin{reptheorem}{subg3}
Suppose that $A,C\in\mathcal{G}$, that $\rank{A}=\rank{C}=r$, and that there does not exist an injection $\iota:A\hookrightarrow C$ such that $\coker{(\iota)}$ is elementary abelian.
Then $S(A,C)=0$.
\end{reptheorem}

\section{Proofs of main results} \label{Proofs}

The combinatorics of the proofs that follow will rely on Lemmas~\ref{subg} to~\ref{subg2.5}, which follow immediately from \hyperref[inj]{Proposition~\ref*{inj}} below.
There are many descriptions of the quantity described in \hyperref[inj]{Proposition~\ref*{inj}}; one such can be found in Theorem~8 in a recent paper of Delaunay and Jouhet~\cite{DJ}.
The formula we present below is different than theirs; hopefully the ease with which it implies Lemmas~\ref{subg} to~\ref{subg2.5} makes up for its unwieldiness.
Before we begin, we introduce some notation.

\begin{Not22} \label{Not22}
Suppose $A\in\mathcal{G}$.
Let $\Lambda\lp A\rp$ be the set of alternating bilinear forms on $A$, with $A$ thought of as a $\lp\Z/\exp{(A)}\rp$-module.
Next, for any $A,B\in\mathcal{G}$, let $\Inj{\lp A,B\rp}$ be the set of injective group homomorphisms from $A$ into $B$.
\end{Not22}

\begin{surjinj}\label{surjinj}
In \hyperref[introduction]{Section~\ref*{introduction}}, we defined moments in terms of surjections, which is standard, but there is an equivalent definition given in terms of injections; see Section~3 of~\cite{G} for more details.
\end{surjinj}


\begin{inj} \label{inj}
Suppose $A=\bigoplus_{i=1}^r{\Z/\ell^{a_i}}$ and $B=\bigoplus_{i=1}^{r^\prime}{\Z/\ell^{b_i}}$, with $a_i\geq a_j$ and $b_i\geq b_j$ for $i\leq j$.
Then
\[
\lv\Inj{\lp A,B\rp}\rv
=\lv\Lambda(A)\rv
\cdot\prod_{i=1}^r{\lp\ell^{\sum_{j=i}^{r^{\prime}}{\min{\{a_i,b_j\}}}}
-\ell^{\sum_{j=i}^{r^{\prime}}{\min{\{a_i-1,b_j\}}}}\rp},
\]
so
\[
\sub{\lp A,B\rp}
=\prod_{i=1}^{r}{\lp\frac{\ell^{\sum_{j=i}^{r^{\prime}}{\min{\{a_i,b_j\}}}}
-\ell^{\sum_{j=i}^{r^{\prime}}{\min{\{a_i-1,b_j\}}}}}
{\ell^{\sum_{j=i}^{r}{a_j}}
-\ell^{\sum_{j=i}^{r}{\min{\{a_i-1,a_j\}}}}}\rp}.
\]
\end{inj}

Before stating some consequences of \hyperref[inj]{Proposition~\ref*{inj}}, we need a bit more notation.

\begin{Not23} \label{Not23}
For any $A\in\mathcal{G}$ and any $i\in\Z^{\geq0}$, let
\[
A_{\oplus i}
:=A\oplus\overbrace{\lp\Z/\ell\rp\oplus\cdots\oplus\lp\Z/\ell\rp}^{i\text{ times}}.
\]
If $i\geq1$, let
\[
\rank_{\ell^i}{A}:=\dim_{\F_\ell}{\lp\ell^{i-1}A/\ell^iA\rp}.
\]
We will abbreviate $\rank_\ell{A}$ by $\rank{A}$.
\end{Not23}
\noindent As an example, consider the group $A=\Z/\ell^4\oplus\Z/\ell^4\oplus\Z/\ell$.
Then $\rank_{\ell^5}{A}=0$, $\rank_{\ell^4}{A}=\rank_{\ell^3}{A}=\rank_{\ell^2}{A}=2$, and $\rank{A}=3$.
We will use the following three lemmas in the proofs of our main results.

\begin{subg} \label{subg}
Suppose $A,B\in\mathcal{G}$.
If $\rank{B}-\rank{A}=i\geq0$, then
\[
\sub{\lp A,A_{\oplus i}\rp}\cdot\sub{\lp A_{\oplus i},B\rp}=\sub{(A,B)}.
\]
\begin{proof}
Computation following from \hyperref[inj]{Proposition~\ref*{inj}}.
\end{proof}
\end{subg}

\begin{subg2} \label{subg2}
Suppose $A,B\in\mathcal{G}$ and $\rank{A}=\rank{B}$.
If
\[
j\leq\max{\lb i\mid\rank_{\ell^i}{A}=\rank{A}\rb},
\]
then
\[
\sub{\lp A\oplus\Z/\ell^j,B\oplus\Z/\ell^j\rp}
=\sub{(A,B)}.
\]
\begin{proof}
Computation following from \hyperref[inj]{Proposition~\ref*{inj}}.
\end{proof}
\end{subg2}

\begin{subg2.5} \label{subg2.5}
Suppose $A\in\mathcal{G}$.
If $i\in\Z^{\geq0}$, $\rank{A}=r$, and $\bigoplus_{j=1}^r{\lp\Z/\ell^i\rp}\leq A$, then
\[
\sub{\lp\bigoplus_{j=1}^r{\lp\Z/\ell^i\rp},A\rp}=1.
\]
\end{subg2.5}
\begin{proof}
Computation following from \hyperref[inj]{Proposition~\ref*{inj}}.
\end{proof}

We now have the tools to prove Theorems~\ref{temp} and~\ref{subg3}.
For any $A,C\in\mathcal{G}$, \hyperref[temp]{Theorem~\ref*{temp}} concerns the case where $\rank{A}<\rank{C}$, and \hyperref[subg3]{Theorem~\ref*{subg3}} concerns the case where $\rank{A}=\rank{C}$.

\begin{temp} \label{temp}
Suppose that $A,C\in\mathcal{G}$ and $\rank{A}<\rank{C}$.
If there exists $k\in\Z^{>0}$ and $B\in\mathcal{G}$ such that $A\leq B<C$, $\rank{B}=\rank{A}$, and $B_{\oplus k}=C$, then $S(A,C)=S(A,B)\cdot S(B,C)$.
Otherwise, $S(A,C)=0$.
\end{temp}
\begin{proof}
By \hyperref[Not21]{Definition~\ref*{Not21}}, we know $S(A,C)$ is a sum of products of subgroup data---one summand for every $A$-chain with maximum $C$.
Choose some such chain, say $\mathfrak{C}=\lb A_i\rb_{i=1}^j$, where $j\in\Z^{>0}$ and $A=A_0<\cdots<A_j=C$.
Consider the set
\[
M_{\mathfrak{C}}
=\lb j_0\in\lb1,\ldots,j\rb
\mid\text{there is no }k_{j_0}\in\Z^{>0}\text{ such that }A_{j_0}=A_{\oplus k_{j^{\prime}}}\rb.
\]
If $M_{\mathfrak{C}}$ is empty, then the theorem is trivially true since there is some $k\in\Z^{>0}$ such that $C=A_{\oplus k}$.
Thus, suppose it is not empty and let $j^\prime=\min{\lp M_{\mathfrak{C}}\rp}$.

There are two possibilities for the ranks of $A_{j^{\prime}}$ and $A_{j^{\prime}-1}$: either $\rank{\lp A_{j^{\prime}-1}\rp}=\rank{\lp A_{j^{\prime}}\rp}$ or $\rank{\lp A_{j^{\prime}-1}\rp}<\rank{\lp A_{j^{\prime}}\rp}$.
It turns out that summands in the former case cancel out those in the latter.
Indeed, if $\rank{\lp A_{j^{\prime}}\rp}-\rank{\lp A_{j^{\prime}-1}\rp}=k_0>0$, then we know by \hyperref[subg]{Lemma~\ref*{subg}} that
\[
\sub{\lp A_{j^{\prime}-1},\lp A_{j^{\prime}-1}\rp_{\oplus k_0}\rp}
\cdot\sub{\lp\lp A_{j^{\prime}-1}\rp_{\oplus k_0},A_{j^{\prime}}\rp}
=\sub{\lp A_{j^{\prime}-1},A_{j^{\prime}}\rp}.
\]
Thus, $\sub{\lp\mathfrak{C}\rp}$ cancels with another summand in $S(A,B)$, one associated to a chain that is longer than $\mathfrak{C}$ by one subgroup; namely, the chain
\[ \tag{$\mathfrak{C}^\prime$}
A_1
<\cdots
<A_{j^{\prime}-1}
<\lp A_{j^{\prime}-1}\rp_{\oplus k_0}
<A_{j^{\prime}}
<\cdots
<A_j=B.
\]
In contrast to $\mathfrak{C}$, the first subgroup in $\mathfrak{C}^\prime$ that is not of the form $A_{\oplus k}$ for any $k\in\Z^{\geq0}$ has the same rank as its predecessor (ie, $\rank{\lp\lp A_{j^{\prime}-1}\rp_{\oplus k_0}\rp}=\rank{\lp A_{j^{\prime}}\rp}$).

Now suppose that $A_{j^{\prime}}$ and $A_{j^\prime-1}$ had satisfied the other possibility; ie, that $\rank{\lp A_{j^{\prime}-1}\rp}=\rank{\lp A_{j^{\prime}}\rp}$.
If $j^{\prime}>1$, then the summand cancels with a summand whose chain is one shorter.
Specifically, we know by \hyperref[subg]{Lemma~\ref*{subg}} that it cancels with the summand associated to the chain $\mathfrak{C}\setminus\lb A_{j^\prime-1}\rb$.
Thus, the only summands of $S(A,B)$ that that remain are those associated to chains with minimum element the same rank as $A$.
Using this fact, we can write
\[
S(A,C)
=-\sum_{\substack{
B_0\in\mathcal{G},A<B_0<C,\\
\rank{B_0}=\rank{A}}}
{\sub{\lp A,B_0\rp}\cdot S\lp B_0,C\rp}.
\]
Note that if $\lb B_0\in\mathcal{G}\mid A<B_0<C,\rank{B_0}=\rank{A}\rb=\emptyset$, then the above sum vanishes and we are done.
Thus, suppose it is not empty and let
\[
B=\max{\lb B_0\in\mathcal{G}\mid A<B_0<C,\rank{B_0}=\rank{A}\rb}.
\]
We can repeat the argument above to see that
\[
S(A,C)
=S(A,B)\cdot S(B,C).
\]
If there is some $k\in\Z^{>0}$ such that $C=B_{\oplus k}$, then we are done.
If not, then the argument from the previous paragraphs and the definition of $B$ imply that $S(B,C)=0$, completing the proof.
\end{proof}

In the light of \hyperref[temp]{Theorem~\ref*{temp}}, we now address $S(A,C)$ in the case where $\rank{A}=\rank{C}$.

\begin{subg3} \label{subg3}
Suppose that $A,C\in\mathcal{G}$, that $\rank{A}=\rank{C}=r$, and that there does not exist an injection $\iota:A\hookrightarrow C$ such that $\coker{(\iota)}$ is elementary abelian.
Then $S(A,C)=0$.
\end{subg3}
\begin{proof}
Suppose that $A<C$ (otherwise the result is trivial).
We will induct on $r$.
To begin, suppose that $r=1$, and define $a,c\in\Z^{\geq0}$ by $\ell^a=\lv A\rv$ and $\ell^c=\lv C\rv$.
Since $A$ and $C$ are cyclic, we see that for any $i\in\lb1,\ldots,c-a\rb$,
\[
\lv\lb A\text{-chains }\mathfrak{C}\mid\max{\mathfrak{C}}=C,\lv\mathfrak{C}\rv=i\rb\rv
=\binom{c-a-1}{i-1}.
\]
Moreover, the fact that $A$ and $C$ are cyclic also implies that $\sub{\lp\mathfrak{C}\rp}=(-1)^i$ for any $A$-chain in the above set.
By assumption, we know that $c-a-1>0$, so $\sum_{i=1}^{c-a}{(-1)^i\binom{c-a-1}{i-1}}=0$.
This completes the base case.

We split the general case into three cases.
For the first case, suppose that $\exp{A}=\exp{C}$.
For any $B\in\mathcal{G}$, let $\overline{B}$ denote $B/\langle b\rangle$, where $b\in B$ is any element of order $\exp{B}$.
Similarly, if $\mathfrak{C}$ is a $B$-chain, we define $\overline{\mathfrak{C}}$ to be $\lb\overline{D}\mid D\in\mathfrak{C}\rb$.
Now, since $\exp{A}=\exp{C}$, we see that $\lb A\text{-chains }\mathfrak{C}\mid\max{\mathfrak{C}}=C\rb$ is in bijection with $\lb\overline{A}\text{-chains }\mathfrak{C}\mid\max{\mathfrak{C}}=\overline{C}\rb$ under the map $\mathfrak{C}\mapsto\overline{\mathfrak{C}}$.
Moreover, given any $A$-chain $\mathfrak{C}$ with $\max{\mathfrak{C}}=C$, \hyperref[inj]{Proposition~\ref*{inj}} implies that $\sub{\lp\mathfrak{C}\rp}=K\cdot\sub{\lp\overline{\mathfrak{C}}\rp}$, where $K$ is a constant depending only on $A$ and $C$.
The result now follows by induction.

For the second case, suppose that $\ell\cdot\exp{A}=\exp{C}$.
For any $A$-chain $C$ with $\max{\mathfrak{C}}=C$, let $\widehat{\mathfrak{C}}$ denote $\min{\lb B\in\mathfrak{C}\mid\exp{B}=\exp{C}\rb}$.
For any $B\in\mathcal{G}$ such that $A<B\leq C$ and $\exp{B}>\exp{A}$, let $B_C=B/\lp\ell^{-1}\exp{C}\rp B$.
For any such $B$ with $B_C\neq A$, we can partition the set $\lb A\text{-chains }\mathfrak{C}\mid\max{\mathfrak{C}=C,\widehat{\mathfrak{C}}=B}\rb$ into two subsets: those chains that contain $B_C$ and those that do not.
We remark that these two subsets are in bijection under the following map: if an $A$-chain does not contain $B_C$, then add it.
The inverse to this map is simply the deletion of $B_C$ from any $A$-chain.
Now, by Lemmas~\ref{subg2} and~\ref{subg2.5}, we know that $B$ has exactly one subgroup isomorphic to $B_C$.
Thus, for any $A_0$ such that $A\leq A_0<B_C$, we know that
\[
\sub{\lp A_0,B_C\rp}\sub{\lp B_C,B\rp}=\sub{\lp A_0,B\rp}.
\]
But this means that any summand associated to a chain in the first subset cancels with the summand associated to the image of the chain under the above bijection.
Thus,
\[
S(A,C)=\sum_{\substack{A\text{-chains }\mathfrak{C}\\
\lp\widehat{\mathfrak{C}}\rp_C=A
}}{\sub{\lp\mathfrak{C}\rp}}.
\]
Now, for any $B\in\mathcal{G}$ with $A<B\leq C$, $\exp{B}>\exp{A}$, and $B_C=A$, note that
\[
\sum_{\substack{A\text{-chains }\mathfrak{C}\\
\max{\mathfrak{C}}=C\\
\widehat{\mathfrak{C}}=B
}}{\sub{\lp\mathfrak{C}\rp}}
=S(B,C)\cdot
\hspace{-40px}\sum_{\substack{
A\text{-chains }\mathfrak{C}\\
\max{\mathfrak{C}}=B\\
\max{\lb\exp{D}\mid D\in\mathfrak{C}\setminus\{B\}\rb}=\exp{A}
}}{\hspace{-30px}\sub{\lp\mathfrak{C}\rp}}.
\]
But $S(B,C)=0$ for all such $B$, by the argument of the previous paragraph.
Thus,
\[
S(A,C)=\sum_{\substack{A\text{-chains }\mathfrak{C}\\
\lp\widehat{\mathfrak{C}}\rp_C=A
}}{\sub{\lp\mathfrak{C}\rp}}
=\sum_{\substack{A<B\leq C\\
\exp{B}>\exp{A}\\
B_C=A
}}{\sum_{\substack{A\text{-chains }\mathfrak{C}\\
\max{\mathfrak{C}}=C\\
\widehat{\mathfrak{C}}=B
}}{\sub{\lp\mathfrak{C}\rp}}}
=0,
\]
completing this case.

Finally, consider the case where $\ell\cdot\exp{A}<\exp{C}$.
As in the previous case, we have that 
\[
S(A,C)=\sum_{\substack{A\text{-chains }\mathfrak{C}\\
\lp\widehat{\mathfrak{C}}\rp_C=A
}}{\sub{\lp\mathfrak{C}\rp}}.
\]
The difference in this case is that if $\mathfrak{C}$ is an $A$-chain with $\max{\mathfrak{C}}=C$, then it is impossible that $\lp\widehat{\mathfrak{C}}\rp_C=A$, so the proof is complete.

\end{proof}

Theorems~\ref{temp} and~\ref{subg3} immediately imply the following corollary.

\begin{final} \label{final}
If $A,C\in\mathcal{G}$, then $S(A,C)=0$ unless there exists an injection $\iota:A\hookrightarrow B$ with $\coker{(\iota)}$ elementary abelian.
\end{final}
\noindent Finally, \hyperref[amalgam]{Remark~\ref*{amalgam}} and \hyperref[final]{Corollary~\ref*{final}} imply the following result.

\begin{trivialyes} \label{trivialyes}
If $A,C\in\mathcal{G}$ and $C$ has exactly one subgroup isomorphic to $A$, then $\mu_C(A,C)=0$ unless there exists some $\iota:A\hookrightarrow C$ with $\coker{(i)}$ elementary abelian.
\end{trivialyes}

\bibliography{mobius}

\newcommand{\etalchar}[1]{$^{#1}$}
\providecommand{\bysame}{\leavevmode\hbox to3em{\hrulefill}\thinspace}
\providecommand{\MR}{\relax\ifhmode\unskip\space\fi MR }
\providecommand{\MRhref}[2]{%
  \href{http://www.ams.org/mathscinet-getitem?mr=#1}{#2}
}
\providecommand{\href}[2]{#2}
\begin{thebibliography}{{Mat}14}

\bibitem[CKL{\etalchar{+}}]{CKLPM}
J.~{Clancy}, N.~{Kaplan}, T.~{Leake}, S.~{Payne}, and M.~{Matchett Wood},
  \emph{\href{http://arxiv.org/abs/1402.5129}{On a {C}ohen-{L}enstra heuristic
  for {J}acobians of random graphs}}, to appear in the
  \href{http://link.springer.com/journal/10801}{Journal of Algebraic
  Combinatorics}.

\bibitem[CL84]{CL}
H.~Cohen and H.~W. Lenstra, Jr.,
  \emph{\href{http://dx.doi.org/10.1007/BFb0099440}{Heuristics on class groups
  of number fields}}, Number theory, {N}oordwijkerhout 1983 ({N}oordwijkerhout,
  1983), Lecture Notes in Math., vol. 1068, Springer, Berlin, 1984, pp.~33--62.
  \MR{756082 (85j:11144)}

\bibitem[CM90]{CM}
Henri Cohen and Jacques Martinet,
  \emph{\href{http://dx.doi.org/10.1515/crll.1990.404.39}{{\'{E}}tude
  heuristique des groupes de classes des corps de nombres}}, J. Reine Angew.
  Math. \textbf{404} (1990), 39--76. \MR{1037430 (91k:11097)}

\bibitem[DJ14]{DJ}
Christophe Delaunay and Fr{\'e}d{\'e}ric Jouhet,
  \emph{\href{http://dx.doi.org/10.1016/j.aim.2014.02.033}{{$p^\ell$}-torsion
  points in finite abelian groups and combinatorial identities}}, Adv. Math.
  \textbf{258} (2014), 13--45. \MR{3190422}

\bibitem[EVW09]{EVW}
J.~S. {Ellenberg}, A.~{Venkatesh}, and C.~{Westerland},
  \emph{{\href{http://arxiv.org/abs/0912.0325}{Homological stability for
  Hurwitz spaces and the Cohen-Lenstra conjecture over function fields}}},
  ArXiv e-prints (2009).

\bibitem[FW89]{FW}
Eduardo Friedman and Lawrence~C. Washington, \emph{On the distribution of
  divisor class groups of curves over a finite field},
  \href{http://www.degruyter.com/view/product/11496}{Th\'eorie des nombres}
  ({Q}uebec, {PQ}, 1987), de Gruyter, Berlin, 1989, pp.~227--239. \MR{1024565
  (91e:11138)}

\bibitem[Gar]{Gab}
Derek Garton, \emph{\href{http://arxiv.org/abs/1405.5824}{Some finite abelian
  group theory and some $q$-series identities}}, to appear in the
  \href{http://link.springer.com/journal/26}{Annals of Combinatorics}.

\bibitem[Gar15]{G}
\bysame, \emph{\href{http://dx.doi.org/10.2140/ant.2015.9.149}{Random matrices,
  the {C}ohen-{L}enstra heuristics, and roots of unity}}, Algebra Number Theory
  \textbf{9} (2015), no.~1, 149--171. \MR{3317763}

\bibitem[Hal34]{Hall2}
P.~Hall, \emph{\href{http://dx.doi.org/10.1112/plms/s2-36.1.29 }{A
  {C}ontribution to the {T}heory of {G}roups of {P}rime-{P}ower {O}rder}},
  Proc. London Math. Soc. \textbf{S2-36} (1934), no.~1, 29. \MR{1575964}

\bibitem[Hal38]{Hall}
\bysame, \emph{\href{http://dx.doi.org/10.1007/BF01199694}{A partition formula
  connected with {A}belian groups}}, Comment. Math. Helv. \textbf{11} (1938),
  no.~1, 126--129. \MR{1509594}

\bibitem[HI{\"O}89]{HIO}
T.~Hawkes, I.~M. Isaacs, and M.~{\"O}zaydin,
  \emph{\href{http://dx.doi.org/10.1216/RMJ-1989-19-4-1003}{On the {M}\"obius
  function of a finite group}}, Rocky Mountain J. Math. \textbf{19} (1989),
  no.~4, 1003--1034. \MR{1039540 (90k:20046)}

\bibitem[Luc07]{Luc}
A.~Lucchini, \emph{\href{http://dx.doi.org/10.1515/JGT.2007.047}{Subgroups of
  solvable groups with non-zero {M}\"obius function}}, J. Group Theory
  \textbf{10} (2007), no.~5, 633--639. \MR{2352034 (2008h:20028)}

\bibitem[Mal10]{Mal}
Gunter Malle, \emph{\href{http://dx.doi.org/10.1080/10586458.2010.10390636}{On
  the distribution of class groups of number fields}}, Experiment. Math.
  \textbf{19} (2010), no.~4, 465--474. \MR{2778658 (2011m:11224)}

\bibitem[{Mat}14]{MMW}
M.~{Matchett Wood}, \emph{\href{http://arxiv.org/abs/1402.5149}{The
  distribution of sandpile groups of random graphs}}, ArXiv e-prints (2014).

\bibitem[Rot64]{Rota}
G.-C. Rota, \emph{\href{http://dx.doi.org/10.1007/BF00531932}{On the
  foundations of combinatorial theory. {I}. {T}heory of {M}\"obius functions}},
  Z. Wahrscheinlichkeitstheorie und Verw. Gebiete \textbf{2} (1964), 340--368
  (1964). \MR{0174487 (30 \#4688)}

\end{thebibliography}
\bibliographystyle{amsalpha}

\end{document}